\documentclass{article}
\usepackage[english]{babel}
\usepackage[utf8]{inputenc}
\usepackage[T1]{fontenc}
\usepackage[normalem]{ulem}
\usepackage{amssymb}
\usepackage{amsmath,amsfonts}
\usepackage[top=3cm,bottom=3cm,left=2.5cm,right=2.5cm]{geometry}
\usepackage{amssymb}
\usepackage{enumerate}
\usepackage[breaklinks]{hyperref}
\usepackage{cleveref}
\usepackage{graphicx} 
\usepackage{amsthm}

\usepackage{authblk}
\usepackage{tikz}
\usepackage{caption}
\usepackage{subcaption}
\theoremstyle{plain}
\newtheorem{theorem}{Theorem}[section]

\newtheorem{proposition}[theorem]{Proposition}
\newtheorem{lemma}[theorem]{Lemma}

\theoremstyle{definition}

\newcommand{\cM}{\mathcal{M}}

\newcommand{\cR}{\mathcal{R}}

\newcommand{\dtv}{d_{\mathrm{TV}}}
\newcommand{\Po}{\mathrm{Po}}
\newcommand{\e}{\mathrm{e}}

\newcommand{\GW}{\mathrm{GW}}

\title{Universality of the matching number in percolated regular graphs}
\author[1]{Sahar Diskin}
\author[2]{Mihyun Kang}
\author[3]{Lyuben Lichev}
\affil[1]{School of Mathematical Sciences, Tel Aviv University, 6997801 Tel Aviv, Israel}
\affil[2]{Institute of Discrete Mathematics, Graz University of Technology, 8010 Graz, Austria}
\affil[3]{Institute of Statistics and Mathematical Methods in Economics, Technical University of Vienna, A-1040 Vienna, Austria}

\date{\today}

\begin{document}

\maketitle

\begin{abstract}
Fix a sequence of $d$-regular graphs $(G_d)_{d\in \mathbb{N}}$ and denote by $G_{d,p}$ the graph obtained from $G_d$ after edge-percolation with probability $p=c/d$, for a constant $c>0$. We prove a quantitative local convergence of $(G_{d,p})_{d\in \mathbb{N}}$.
In combination with results of Bordenave, Lelarge and Salez, it implies that the rescaled matching number of $G_{d,p}$ is asymptotically equivalent to that of the binomial random graph $G(n,c/n)$.
\end{abstract}

\section{Introduction} 
A \textit{matching} in a graph $G$ is a collection of pairwise non-adjacent edges $M\subseteq E(G)$. 
Denote by $\cM(G)$ the set of all matchings in $G$, and define the \textit{matching number} of $G$ by
\begin{align*}
    \nu(G)~:=~\max_{M\in\cM(G)}|M|.
\end{align*}
A graph is said to have a \textit{perfect matching} if it has a matching $M$ incident to all vertices of $G$; equivalently, when $2\nu(G)=|V(G)|$. 

The \textit{binomial random graph} $G(n,p)$ is formed by retaining each edge of the complete graph on $n$ vertices $K_n$ independently with probability $p$. When $(n-1)p=c$ for some constant $c>0$, $G(n,p)$ typically contains isolated vertices, and thus cannot have a perfect matching. Nonetheless, determining the asymptotic value of the rescaled matching number $\nu(G(n,p))/n$ is not a trivial task. For a constant $c>0$, let $y=y(c)$ be the smallest solution in $(0,1)$ of
\begin{align}
    y~=~\exp\left(-c\exp(-cy)\right), \label{eq: definition of y}
\end{align}
and consider a function 
\begin{align}
    F(c) ~:=~ 1-\frac{y+\exp(-cy)+cy\exp(-cy)}{2}.\label{eq: definition of F}
\end{align}
We note that $\lim_{c\to \infty} F(c)=1/2$. 
In a seminal work, Karp and Sipser \cite{KS81} showed that
\begin{align*}
    \frac{\nu\left(G(n,c/n)\right)}{n}\quad\xrightarrow[n\to\infty]{}\quad(1+o(1))~F(c),
\end{align*}
where the convergence holds in probability. 
At the heart of the proof of this result lies the so-called \textit{Karp-Sipser algorithm}. It is based on the observation that any vertex of degree one, called a leaf, participates in at least one maximum matching.
The algorithm applies an iterative deletion process of leaves and their unique neighbours, ending up with the \emph{Karp-Sipser core} of the graph. 
The careful analysis of the Karp-Sipser algorithm  provides precise asymptotic expressions for the rescaled matching number and the rank of the adjacency matrix of random graphs~\cite{AFP98,BF11,BC24,BCC22, COparity23,COrank23,GKSS24,Kre17}. 

The binomial random graph $G(n,p)$ is one instance of \textit{edge-percolation}.
Given a host graph $G$ and a probability $p\in[0,1]$, we form the random subgraph $G_p\subseteq G$ by retaining each edge of $G$ independently and with probability $p$. (In particular, $G(n,p)$ is obtained by percolating the edges of the complete graph $K_n$ with probability $p$.) Another well-studied instance of edge-percolation is the \textit{percolated hypercube} $Q^d_p$. The $d$-dimensional binary hypercube $Q^d$ is the graph with vertex set $\{0,1\}^d$ where $uv\in E(Q^d)$ if and only if $u$ and $v$ differ in a single coordinate. 
Similarly to the binomial random graph, in the $p$-percolated hypercube, a giant component suddenly arises around the point $p=1/d$ (see \cite[Chapter 13]{FK16}). In this note, we extend this analogy to the matching number of $Q^d_p$.

A simple argument shows that $\nu(Q^d_{c/d})/2^d$ tends to $1/2$ as $c$ tends to infinity (see \cite[Theorem~1]{DEKK25}). 
Yet, obtaining a precise result similar to that of Karp and Sipser \cite{KS81} is not immediate: indeed, in their proof, the underlying analysis based on differential equations uses the homogeneity of the complete graph. 
In particular, the high-dimensional lattice-like geometry of $Q^d$ invalidates this approach. 

Our main result establishes an analogue of the Karp-Sipser result for the percolated $d$-dimensional hypercube and, in fact, for any percolated $d$-regular graph with growing degree.
\begin{theorem}\label{thm:main}
Fix a constant $c>0$ and a sequence of $d$-regular graphs $(G_d)_{d\in \mathbb{N}}$. Let $p = p(d) = c/d$. Then,
\begin{align*}
    {\frac{\nu\left((G_d)_p\right)}{|V(G_d)|}}\quad\xrightarrow[d\to \infty]{}\quad  F(c) \hspace{10ex} \textit{in probability,}
\end{align*}
where $F(c)$ is the same as in \eqref{eq: definition of F}.
\end{theorem}

The main part of the proof of \Cref{thm:main} is based on showing that percolated regular graphs with growing degree  converge locally in probability to a Galton-Watson branching process with offspring distribution $\Po(c)$ (see Theorem~\ref{thm:local_limit}) which, together with a result of Bordenave, Lelarge, and Salez~\cite{BLS13} (Theorem~\ref{thm:BLS}), implies Theorem~\ref{thm:main}.

Theorem \ref{thm:main} contributes to an ongoing line of research, dedicated to the study of universality of binomial random graphs. Previous research has mainly studied the typical emergence of the giant component, its uniqueness, and asymptotic properties \cite{CDE24, DEKK23, DEKK24, L22}.
One significant feature distinguishing the above setting from our analysis of the matching number is that the minimal assumptions of regularity and growing degree in Theorem \ref{thm:main} are insufficient to uniquely identify the size of the giant component.
For example, no giant component can appear as a subgraph of a disjoint union of cliques of order $d+1$. In fact, the emergence of a giant component -- and the universality of this phenomenon -- are tightly connected to the expansion properties of the host graph. 
In contrast, our results show that the universality class for the matching number is much broader and contains \textit{all} regular graphs of growing degree. Finally, note that the assumption that $d$ is growing is \textit{necessary} for our result to hold in such generality, as percolated regular graphs of bounded degree have a different local limit.

This note is structured as follows. In Section \ref{sec:notation}, we introduce some notation and terminology. In Section~\ref{sec:3}, we prove a quantitative strengthening of Theorem \ref{thm:main}. Finally, in Section \ref{sec:conc}, we briefly discuss extensions of Theorem~\ref{thm:main} to even more general settings.

\section{Notation and terminology}\label{sec:notation}
In this note, $\mathbb N$ denotes the set of positive integers. 
For $0<a<b$, we write $[a,b]=\{a,a+1,\ldots,b\}$ instead of $\{a,a+1,\ldots,b\}\cap \mathbb N$ for brevity, and we also write $[k]=\{1,2,\ldots,k\}$  for all $k\in \mathbb{N}$. All logarithms in this paper have base $\e$. 
Given a graph $G$ and a set $U\subseteq V(G)$, we denote by $\Delta(G)$ the maximum degree of $G$ and by $G[U]$ the subgraph of $G$ induced by $U$.

Let us further recall several variants of the notion of local convergence of a sequence of graphs. For a more detailed account on the topic, see~\cite{Sal11} and Sections~2.3 and~2.4 in~\cite{Hof24}. 

Given a graph $H$, a vertex $v$ in $H$ and an integer $r\ge 0$, the \emph{$r$-th (closed) neighbourhood of $v$ in $G$}, denoted $N^r_H[v]$, is the set of vertices at (graph-)distance at most $r$ from $v$ in $H$.
Moreover, the graph induced by these vertices is called the \emph{ball with radius $r$ around $v$ in $H$} and denoted by $B^r_H[v]$.
We often see $B^r_H[v]$ as a graph rooted at $v$. Given a graph $H$ and a vertex $v$, we denote by $C_H(v)$ the connected component of $H$ containing $v$. Then, with a slight abuse of notation, we think of the rooted graph $(H,v)$ as the rooted (connected) graph $(C_H(v),v)$.
We can then denote the set of finite rooted graphs by $\cR$ and equip it with the natural distance
\[d\left((H_1,v_1), (H_2,v_2)\right) ~:=~ \frac{1}{1+\sup\{r\in \mathbb N: B^r_{H_1}[v_1]\simeq B^r_{H_2}[v_2]\}},\]
where $d((H_1,v_1),(H_2,v_2))=1$ means that the connected components of $v_1$ in $H_1$ and of $v_2$ in $H_2$ are isomorphic as rooted graphs.

Fix a sequence of (deterministic or random) finite rooted graphs $(H_n,u_n)_{n\ge 1}$, where $u_n$ is a vertex of $H_n$ chosen uniformly at random.
Given a probability measure $\mu$ on the set $\cR$, we say that $(H_n,u_n)_{n\ge 1}$ \emph{converges locally in probability} to $\mu$ if, for every $r\in\mathbb{N}$ and every finite rooted graph~$(\bar{H},\bar{u})$,
\[\mathbb P(B^r_{H_n}[u_n] \simeq (\bar{H},\bar{u})\mid (H_n,u_n))\quad\xrightarrow{}\quad \mu(\{(H,u)\in \cR: B^r_H[u] \simeq (\bar{H}, \bar{u})\})  \hspace{10ex} \text{in probability},\]
or equivalently if the sequence of random variables $\mathbb E[h(H_n,u_n)\mid (H_n,u_n)]$ converges in probability to the integral of $h$ with respect to $\mu$ for every bounded and continuous function $h: \cR\to \mathbb R$.
The limiting measure $\mu$ is often identified with a (possibly random) rooted graph.

In the sequel, we will provide a stronger quantitative version of local convergence in terms of the \emph{total variation distance} defined as follows: for two probability distributions $\mu_1$ and $\mu_2$ defined on a common probability space $\Omega$, the total variation distance between $\mu_1$ and $\mu_2$ is defined as 
\begin{equation*}
\begin{split}
     \dtv(\mu_1,\mu_2)
     &~:=~ \sup_{A} |\mu_1(A)-\mu_2(A)|\\
     &~=~\inf\{\sup_{A}\mathbb P(X\in A, Y\notin A): (X,Y)\text{ coupling of }(\mu_1,\mu_2)\},
 \end{split}
\end{equation*} 
where the suprema are taken over a collection of measurable events generating the underlying $\sigma$-algebra (in many cases, the cylindric events). Equivalently, when $\Omega$ is a finite or countable space, we have
\begin{equation}\label{eq:dTV}
\dtv(\mu_1,\mu_2) ~=~ \frac{1}{2}\sum_{\omega\in \Omega} |\mu_1(\omega) - \mu_2(\omega)|.
\end{equation}

\section{\texorpdfstring{Proof of Theorem \ref{thm:main}}{}}\label{sec:3}

Throughout this note, we fix a constant $c>0$, set $p = p(d) := c/d$, and let $(G_d)_{d\in \mathbb{N}}$ be a sequence of $d$-regular graphs as in Theorem \ref{thm:main}. 
Percolating the edges of the graph $G_d$ with probability $p$ gives rise to a spanning random subgraph $G_{d,p}=(G_d)_p$ of $G_d$.
For every $r\in\mathbb{N}$, we define a random measure $\mu_{r,d}$ so that, for every $(H,u)\in \cR$, conditionally on $G_{d,p}$, 
\[\mu_{r,d}(H,u) ~:=~ \frac{1}{|V(G_{d,p})|}\ \big|\{v\in G_{d,p}: B^r_{G_{d,p}}[v]\simeq (H,u)\}\big|.\]
Denote by $\GW_c$ the Galton-Watson tree rooted at a vertex $o$ and with offspring distribution $\Po(c)$. For every $r\in\mathbb{N}$ and $(H,u)\in \cR$, let $\mu_r$ be the distribution of $B^r_{\GW_c}[o]$. In other words,  
\[\mu_{r}(H,u) ~:=~ \mathbb P\left( B^r_{\GW_c}[o] \simeq (H, u) \right).\]

The following theorem is the central technical result in this note.

\begin{theorem}\label{thm:local_limit}
For every $r\in\mathbb{N}$ and every sufficiently large $d\in\mathbb{N}$,
\[\mathbb P\left(\dtv(\mu_{r,d}, \mu_r)\ge \exp\left(-\frac{1}{4}(\log d)^{1/2r}\right)\right) =o(d^{-1/3}).\] 
In particular, $(\mu_{r,d})_{d\in \mathbb{N}}$ converges locally in probability to $\mu_r$.
\end{theorem}

\subsection{Coupling estimates}

Before diving into the proof of Theorem~\ref{thm:local_limit}, we provide a couple of probabilistic bounds. The first lemma is a simple consequence of Chernoff's bound, which will be of use throughout the note.

\begin{lemma}\label{lem:Chernoff}
For every sufficiently large $d\in\mathbb{N}$, every $d'\in[d-d^{1/4},d]$ and every $t \ge 10c$,
\[\mathbb P(\mathrm{Bin}(d',p)\ge t) \le \e^{-t/3}\quad \text{and}\quad \mathbb P(\Po(c)\ge t)  \le \e^{-t/3}.\]
\end{lemma}

The following lemma quantifies the total variation distance between the binomial and the Poisson distributions for a suitable choice of parameters. We include its proof for the sake of completeness.

\begin{lemma}\label{lem:Bin vs Po}
For every sufficiently large $d\in \mathbb{N}$ and every $d'\in[d-d^{1/4},d]$,
\[\dtv\left(\mathrm{Bin}(d',p), \Po(c)\right) = O(d^{-1/2}).\]
Here, we artificially extend the support of $\mathrm{Bin}(d',p)$ to the set of non-negative integers by giving weight $0$ to each element of the set $\{d'+1,d'+2,\ldots\}$.
\end{lemma}
\begin{proof}
By~\eqref{eq:dTV}, for all $d\in\mathbb{N}$ such that $d-d^{1/4}\ge d^{1/4}$ and $d'\in[d-d^{1/4},d]$, we have
\begin{align}
\dtv(\mathrm{Bin}(d',p), \Po(c)) 
&= \frac{1}{2}\sum_{i=0}^{\infty} \bigg|\binom{d'}{i} p^i (1-p)^{d'-i} - \e^{-c}\frac{c^i}{i!}\bigg|\nonumber\\ 
\begin{split}\label{eq:4}
&\le\frac{1}{2} \sum_{i=0}^{d^{1/4}} \frac{1}{i!} \bigg|(1-p)^{d'-i}\bigg(\prod_{j=0}^{i-1} p(d'-j)\bigg) - \e^{-c}c^i\bigg|\\
&\quad\quad\quad\quad\quad+\mathbb P(\mathrm{Bin}(d',p)>d^{1/4})+\mathbb P(\Po(c)>d^{1/4}),
\end{split}
\end{align}
where we use the conventions that $\binom{d'}{i}=0$ for every $i\in \{d'+1,d'+2,\ldots\}$ and the empty product is equal to one, that is, $\prod_{i=0}^{-1}x=1$ for any $x\in\mathbb{R}$.

First, note that the standard Chernoff bound shows that $\mathbb P(\mathrm{Bin}(d',p)>d^{1/4})$ and $\mathbb P(\Po(c)>d^{1/4})$ are both of order $o(d^{-1/2})$. Now, for every $i\in [0,d^{1/4}]$, we have 
\begin{align*}
    (1-p)^{d'-i}=\exp\left(-p(d'-i)+O(p^2(d'-i))\right)=\exp\left(-c+O(d^{-3/4})\right)=(1+O(d^{-3/4}))\e^{-c},
\end{align*}
and
\begin{align*}
    \prod_{j=0}^{i-1}p(d'-j)&=\prod_{j=0}^{i-1}c\left(1-\frac{j}{d}+O(d^{-3/4})\right)\\
    &=c^i\exp\left(-\frac{i(i-1)}{2d}+O(i^3d^{-2}+id^{-3/4})\right)=c^i\exp(O(d^{-1/2})) = (1+O(d^{-1/2}))c^i,
\end{align*}
where the constant in the $O$-notation does not depend on $i$.
Altogether, we obtain
\[\sum_{i=0}^{d^{1/4}} \frac{1}{i!} \bigg|(1-p)^{d'-i}\bigg(\prod_{j=0}^{i-1} p(d'-j)\bigg) - \e^{-c}c^i\bigg| \le \sum_{i=0}^{d^{1/4}} \e^{-c}\cdot\frac{c^i}{i!} \cdot O(d^{-1/2}) = O(d^{-1/2}).\]
Combined with~\eqref{eq:4} and the bounds on $\mathbb P(\mathrm{Bin}(d',p)>d^{1/4})$ and $\mathbb P(\Po(c)>d^{1/4})$, this completes the proof.
\end{proof}

\subsection{Quantitative local convergence: proof of Theorem~\ref{thm:local_limit}}
We first show that $(G_{d,p})_{d\in \mathbb{N}}$ converges locally in distribution to $\GW_c$, that is, the Galton-Watson tree rooted at $o$ with offspring distribution $\Po(c)$. 
Let $\GW^r_c$ be the subtree of $\GW_c$ induced by the vertices at distance at most $r$ from its root $o$.

\begin{proposition}\label{prop:weak_conv}
Fix a vertex $u$ in $G_d$. Then, for every $r\in\mathbb{N}$ and sufficiently large $d\in\mathbb{N}$,
\[\dtv\left(B^r_{G_{d,p}}[u], \GW^r_c \right)\le (\log d)^r d^{-1/2}.\]
\end{proposition}
\begin{proof}
We couple $(G_{d,p},u)$ and $(\GW_c,o)$ by performing two parallel \textit{breadth-first explorations}, abbreviated BFEs, during which we will gradually reveal subgraphs of $G_{d,p}$ and of $\GW_c$. To this end, we maintain two sequences of random variables: a sequence $(X_e)_{e\in E(G_d)}$ of i.i.d. Bernoulli$(p)$ random variables, and a sequence $(Y_j)_{j=1}^{\infty}$ of i.i.d. $\Po(c)$ random variables. 

The BFE of $G_{d,p}$ is described as follows. We maintain three sets of vertices: the set $P$ of passive vertices (already explored), the set $A$ of active vertices, and the set $U$ of unexplored vertices. 
We initialise $P=\varnothing, A=\{u\},$ and $U=V(G_d)\setminus\{u\}$. The set $A$ is processed as a queue, that is, according to a first-in-first-out order. 
At every step, if $A$ is not empty, consider the first vertex $v$ in $A$ together with its neighbours $v_1,\ldots, v_k$ in $G_d[U\cup \{v\}]$. We retain the edges $(vv_i)_{i=1}^k$ according to $(X_{vv_i})_{i=1}^k$, respectively. 
Assume that $v_1',\ldots,v_j'$ are the neighbours of $v$ established by the BFE. 
Then, we move $v$ from $A$ to $P$, move $v_1',\ldots,v_j'$ from $U$ to $A$, and proceed to the next step. Note that edges between vertices in $A$ are not revealed during this process. 
The algorithm terminates once $A$ is empty.

The BFE of $\GW_c$ follows similar lines but we maintain only two sets, $P'$ and $A'$. We initialise $P'=\varnothing$ and $A'=\{o\}$. 
At the $j$-th step, if $A'$ is not empty, consider the first vertex $v$ in $A'$. Then, add $Y_j$ many new vertices to $A'$, and move $v$ from $A'$ to $P'$.
Once again, the algorithm terminates once $A'$ is empty.

It remains to couple the two BFEs. 
Initially (at step $0$), we couple $u$ with $o$. 
Fix $i\in \mathbb{N}$ and suppose that the first $i-1$ steps of the two BFEs have been coupled. Consider the first vertices $v\in A$ and $w\in A'$. 
We continue as follows. If the number of neighbours of $v$ in $G_d[U\cup \{v\}]$ is less than $d-d^{1/4}$, or if the number of children of $v$ and that of $w$ in the BFE are not the same, then we abort the process.
Otherwise, let $\ell$ be the number of children of $v$ and of $w$. If $\ell\ge 4\log d$, we abort the process. Otherwise, we pair these $\ell$ children of $v$ and of $w$ according to the BFE order and proceed to the next step. 

We now estimate the probability that the process terminates at step $i\in \mathbb{N}$. Let $v$ be the first vertex in $A$, and let $w$ be the first vertex in $A'$.
As the process was not aborted at previous steps, at this moment, we have $|P\cup A|\le i\cdot 4\log d$. Thus, as long as $i< d^{1/4}(4\log d)^{-1}$, we have that $v$ has at least $d-d^{1/4}$ neighbours in $G_d[U\cup \{v\}]$. 
Moreover, by Lemma~\ref{lem:Bin vs Po}, the offsprings of $v$ and $w$ can be coupled in a way ensuring that the (global) coupling fails (that is, they do not have the same number of children) with probability $O(d^{-1/2})$ at the current step. 
Finally, upon a successful coupling of the offsprings of $v$ and $w$, by Lemma~\ref{lem:Chernoff}, the probability that the number of children is at least $4\log d$ is $o(1/d)$. Thus, for every $i<d^{1/4}(4\log d)^{-1}$, the coupling terminates at step $i$ with probability $O(d^{-1/2})$. 

Further, for every $r\in\mathbb{N}$, we have that $1+4\log d+\ldots+(4\log d)^{r-1}\le 2(4\log d)^{r-1}$. 
Moreover, upon successful coupling during $2(4\log d)^{r-1}$ steps, the probability that $G_{d,p}$ contains an edge between two vertices which were simultaneously active at some of these steps is at most $p\cdot (2(4\log d)^{r-1})^2\le (\log d)^{2r-1}/d$.
Note that, as long as the coupling of the BFEs is successful for the first $2(4\log d)^{r-1}$ steps and the latter event does not hold, $B^r_{G_{d,p}}[u]$ and $\GW^r_c$ are coupled in a way ensuring that $B^r_{G_{d,p}}[u]\simeq \GW^r_c$.
Putting everything together, we obtain
\[\dtv\left(B^r_{G_{d,p}}[u], \GW^r_c\right)\le 2(4\log d)^{r-1}\cdot O(d^{-1/2}) + (\log d)^{2r-1}/d = o((\log d)^r d^{-1/2}),\]
which implies the statement of the proposition.
\end{proof}

Next, we will strengthen Proposition~\ref{prop:weak_conv} via a suitable concentration argument.
To this end, we utilise the following variant of Azuma-Hoeffding's inequality (see, e.g., Theorem 3.9 in \cite{M98} and Corollary 6 in \cite{War16}).
\begin{theorem}\label{thm:BDI}
Let $Z = (Z_1,\ldots,Z_m)$ be a family of i.i.d. Bernoulli$(p)$ random variables. Fix a constant $k>0$ and a function $f:\{0,1\}^m\to \mathbb R$ such that $|f(x)-f(y)|\le k$ for every $x,y\in \{0,1\}^m$ which differ in a single coordinate.
Then, for every $t\ge 0$,
\[\mathbb P(|f(Z)-\mathbb E f(Z)|\ge t)\le 2\exp\bigg(-\frac{t^2}{2(1-p)pk^2m+2kt/3}\bigg).\]
\end{theorem}

The following lemma is a simple consequence of Theorem~\ref{thm:BDI}.

\begin{lemma}\label{lem:conc}
Fix integers $r\ge 1,\Delta\ge 2$ and fix a rooted tree $(T,o)$ with height at most $r$ and maximum degree at most $\Delta$. 
Denote by $N_{r,T}$ the number of vertices in $G_{d,p}$ whose $r$-th neighbourhood is isomorphic to $(T,o)$ (as rooted graphs). Let $m$ be the number of edges in $G_d$.
Then, 
\[\mathbb P\left(|N_{r,T} - \mathbb E N_{r,T}|\ge (\mathbb E N_{r,T})^{2/3}\right)\le \exp\bigg(-\frac{(\mathbb E N_{r,T})^{4/3}}{9p\Delta^{2r}m}\bigg).\]
\end{lemma}
\begin{proof}
Order the edges of $G_d$ arbitrarily.
For every $i\in [m]$, denote by $Z_i$ the indicator random variables of the event that the $i$-th edge of $G_d$ belongs to $G_{d,p}$.

Fix $e=uv\in E(G_d)$. We count the number of vertices $o'$ in $G_d$ such that the ball with radius $r$ around $o'$ in $G_{d,p}\cup \{e\}$ contains $e$ and is isomorphic to $(T,o)$. 
Note that every such vertex $o'$ must be at distance at most $r-1$ from the closer vertex (with respect to the graph distance in $G_{d,p}$) among $u$ and $v$, say $w\in \{u,v\}$. 
Moreover, since $\Delta(T)\le \Delta$, all vertices on the shortest path between $o'$ and $w$ must have degree at most $\Delta$ in $G_{d,p}$.
As a result, one can choose $o'$ among the vertices reachable from $w\in \{u,v\}$ in $G_{d,p}$ along such paths. In particular, the number of such vertices is at most $2(\Delta+\Delta^2+\ldots+\Delta^{r-1})\le 2\Delta^r$.
Hence, by applying Theorem~\ref{thm:BDI} with $k = 2\Delta^r$ and $t=(\mathbb E N_{r,T})^{2/3}$, we obtain that
\begin{align*}
\mathbb P\left(|N_{r,T} - \mathbb E N_{r,T}|\ge (\mathbb E N_{r,T})^{2/3}\right)
&\le 2\exp\bigg(-\frac{(\mathbb E N_{r,T})^{4/3}}{8(1-p)p\Delta^{2r}m+4\Delta^r(\mathbb E N_{r,T})^{2/3}/3}\bigg)\\
&\le \exp\bigg(-\frac{(\mathbb E N_{r,T})^{4/3}}{9p\Delta^{2r}m}\bigg),
\end{align*}
where the last step follows from the inequality $(\mathbb E N_{r,T})^{2/3}\le |G_d|^{2/3} = o(pm)$.
This concludes the proof.
\end{proof}

We are ready to complete the proof of Theorem~\ref{thm:local_limit}.

\begin{proof}[Proof of Theorem~\ref{thm:local_limit}]
Observe that, for every rooted tree $(T,o)$ with height $r$ and maximum degree at most $\gamma = \gamma(d,r) := (\log d)^{1/2r}$, we have
\begin{align*}
\mathbb P\left(\GW^r_c\simeq (T,o)\right)&\ge \prod_{i=0}^{2\gamma^{r-1}} \bigg(\e^{-c} \frac{c^{\gamma}}{\gamma!}\bigg)\ge \gamma^{-2\gamma^r}\\
&=\exp\left(-2(\log d)^{1/2}\cdot \frac{1}{2r}\log\log d\right)= d^{-o(1)} = \omega((\log d)^r d^{-1/2}),
\end{align*}
where the first inequality uses that there are at most $2\gamma^{r-1}$ vertices whose descendants we need to consider, that the number of descendants of any vertex in $\GW^r_c$ follows a $\Po(c)$ distribution, and it matches the correct value (indicated by $(T,o)$) with probability at least $\mathbb P(\Po(c)=\gamma)$.

Thus, by the last observation together with Proposition~\ref{prop:weak_conv}, for every fixed vertex $u$ in $G_d$, the probability that the ball with radius $r$ around $u$ in $G_{d,p}$ is isomorphic to $(T,o)$ is equal to $\mathbb P(\GW^r_c\simeq (T,o))$ up to lower order terms. 
Moreover, by setting $n = |V(G_d)|$, for every $r\in\mathbb{N}$ and every rooted tree $(T,o)$ as above, we get
\[\mathbb E N_{r,T} = \left(\mathbb P\left(\GW^r_c\simeq (T,o)\right) + O((\log d)^r d^{-1/2})\right)n.\]
By combining the above observations, the triangle inequality and Lemma \ref{lem:conc}, we obtain that, with probability at least $1-\exp(-n^{1/3+o(1)})$,
\begin{equation}\label{eq:T}
\begin{split}
\left|N_{r,T}-\mathbb P\left(\GW^r_c\simeq (T,o)\right)n\right|
&\le |N_{r,T}-\mathbb E N_{r,T}| + |\mathbb E N_{r,T} - \mathbb P(\GW^r_c\simeq (T,o))n|\\
&\le (\mathbb E N_{r,T})^{2/3}+(\log d)^r d^{-1/2}n\le d^{-1/4} n.    
\end{split}
\end{equation}
Since there are at most $\gamma^{2\gamma^{r-1}}$ finite rooted trees with height at most $r$ and maximum degree at most $\gamma$, using the latter probability bound and a union bound over all such trees, with probability at least
\[1 - \gamma^{2\gamma^{r-1}}\cdot \exp(-n^{1/3+o(1)}) \ge 1 - \exp(-d^{1/3+o(1)}),\]
we have that \eqref{eq:T} holds, for every rooted tree $(T,o)$ as above.
Moreover, by using Lemma~\ref{lem:Chernoff}, we obtain
\begin{align*}
\mathbb P(\Delta(\GW^r_c)\ge \gamma)  
&= \mathbb P(\Delta(\GW^1_c) \ge \gamma) + \sum_{i=2}^{r} \mathbb P(\Delta(\GW^i_c)\ge \gamma\mid \Delta(\GW^{i-1}_c) < \gamma)\\
&= \mathbb P(\Po(c)\ge \gamma) + \sum_{i=2}^{r} \gamma^{i-1} \mathbb P(\Po(c)\ge \gamma)\le 2\gamma^{r-1} \e^{-\gamma/3}.
\end{align*}

By combining the previous observations, with probability $1 - \exp(-d^{1/3+o(1)})$, we have
\[\dtv(\mu_{r,d}, \mu_r)\le \gamma^{2\gamma^{r-1}}\cdot d^{-1/4} + 2\mathbb P(\Delta(\GW^r_c)\ge \gamma)\le \e^{-\gamma/4},\]
as desired.
\end{proof}

\subsection{\texorpdfstring{Asymptotic matching number: proof of Theorem \ref{thm:main}}{}}
The last missing piece is a reformulation of Theorem 2 in the work of Bordenave, Lelarge and Salez~\cite{BLS13}.
\begin{theorem}\label{thm:BLS}
Let $(H_n)_{n\in \mathbb{N}}$ be a sequence of graphs converging locally in probability to $\GW_c$. Then,
\begin{align*}
   \frac{\nu(H_n)}{|V(H_n)|}\quad\xrightarrow[n\to \infty]{}\quad F(c) \hspace{10ex} \textit{in probability}
\end{align*}
where $F(c)$ is the same as in \eqref{eq: definition of F}.
\end{theorem}
Theorem \ref{thm:main} now follows by combining \Cref{thm:local_limit} and \Cref{thm:BLS}.

\section{Concluding remarks}\label{sec:conc}

In this note, we show quantitative local convergence in probability of percolated regular graphs and deduce a convergence for the rescaled matching number. 
Our methods can be extended in several directions without significant further effort.
For example, slightly more careful estimates in \Cref{sec:3} allow to extend \Cref{thm:local_limit} to a sequence of {\it approximately} $d$-regular graphs $(G_d)_{d\in \mathbb N}$ whose minimum degree $\delta(G_d)$ and maximum degree $\Delta(G_d)$ satisfy $\delta(G_d), \Delta(G_d) = d\pm o(d)$.
The local limit and the asymptotic value of the rescaled matching number remain unchanged in this case, and the error bound would have to be incorporated in the quantitative estimate in \Cref{thm:local_limit}.
Furthermore, inhomogeneous multi-partite graphs may also be treated similarly.
More precisely, fix $k\in\mathbb{N}$ and a growing sequence of graphs equipped with partitions $(U_{1,d},\ldots,U_{k,d})_{d\in \mathbb{N}}$ of their vertex set. 
Also, fix a sequence of symmetric matrices $(D_d)_{d\in \mathbb{N}}$ and $(F_d)_{d\in \mathbb{N}}$ with entries $(d_{i,j})_{i,j=1}^k$ and $(f_{i,j})_{i,j=1}^k$ where, for every $i,j\in [k]$, $d_{i,j}$ and $f_{i,j}$ are functions of $d$, $d_{i,j}$ tends to infinity and $f_{i,j} = o(d_{i,j})$.
Suppose that every vertex in $U_{i,d}$ is incident to a number of neighbours in $U_{j,d}$ between $d_{i,j} - f_{i,j}$ and $d_{i,j}+f_{i,j}$.
Finally, fix a symmetric matrix $(C_{i,j})_{i,j=1}^k$ and define $p_{i,j} = C_{i,j}/d_{i,j}$.
Then, as $d\to \infty$, for every $r\in\mathbb{N}$, the empirical measure of the $r$-th neighbourhoods of the vertices in the random graph obtained by percolating the edges between $U_{i,d}$ and $U_{j,d}$ with probability $p_{i,j}$ 
(and, in particular, the edges within $U_{i,d}$ with probability $p_{i,i}$) for all $i,j\in [k]$ converges locally in probability to the distribution of the $r$-th neighbourhood of a $k$-type branching process where, for every $i,j\in [k]$, each vertex of type $i\in [k]$ produces a number of children of type $j\in [k]$ distributed as $\Po(C_{i,j})$. 
While this implies the existence of a limit for the rescaled matching number via a more general version of \Cref{thm:BLS} (see~\cite{BLS13}), finding the exact limit requires resolving an optimisation problem depending on the sizes of $(U_{i,d})_{i=1}^k$ and the entries of $(C_{i,j})_{i,j=1}^k$.
The proofs of the above statements introduce no additional challenges but complicate the notation, which is why we restricted our attention to the cleaner statement of Theorem~\ref{thm:local_limit}.

It is worth noting that local limits turn out to be an important tool in the asymptotic analysis of important graph parameters. For example, by formalising heuristics originating from the cavity method, Salez~\cite{Sal13} provided precise expressions for the number of copies of certain spanning subgraphs in large tree-like graphs, see also \cite{Sal11}. Krivelevich, M\'esz\'aros, Michaeli, and Shikhelman~\cite{KMMS24} relied on the local convergence framework to analyse a random greedy construction of independent sets on a variety of random graph models.

\paragraph{Acknowledgements.} The authors are grateful to Michael Anastos for fruitful discussions, to Michael Krivelevich for helpful comments and suggestions, and to Justin Salez for turning our attention to the reference~\cite{Sal13}.
Part of this work was conducted during visits of the first and the third author to TU Graz. 
The second author was supported in part by the Austrian Science Fund (FWF) [10.55776/F1002] and the third author by the Austrian Science Fund (FWF) [10.55776/ESP624]. For open access purposes, the authors have applied a CC BY public copyright license to any author-accepted manuscript version arising from this submission.

\bibliographystyle{abbrv}
\bibliography{Bib}

\begin{thebibliography}{10}

\bibitem{AFP98}
J.~Aronson, A.~Frieze, and B.~G. Pittel.
\newblock Maximum matchings in sparse random graphs: {K}arp-{S}ipser revisited.
\newblock {\em Random Structures and Algorithms}, 12(2):111--177, 1998.

\bibitem{BF11}
T.~Bohman and A.~Frieze.
\newblock Karp-{S}ipser on random graphs with a fixed degree sequence.
\newblock {\em Combinatorics, Probability and Computing}, 20(5):721--741, 2011.

\bibitem{BLS13}
C.~Bordenave, M.~Lelarge, and J.~Salez.
\newblock Matchings on infinite graphs.
\newblock {\em Probability Theory and Related Fields}, 157:183--208, 2013.

\bibitem{BC24}
T.~Budzinski and A.~Contat.
\newblock The critical {K}arp-{S}ipser core of {E}rd{\H{o}}s-{R}{\'e}nyi random graphs.
\newblock {\em arXiv preprint arXiv:2412.04328}, 2024.

\bibitem{BCC22}
T.~Budzinski, A.~Contat, and N.~Curien.
\newblock The critical {K}arp-{S}ipser core of random graphs.
\newblock {\em arXiv preprint arXiv:2212.02463}, 2022.

\bibitem{COparity23}
A.~Coja-Oghlan, O.~Cooley, M.~Kang, J.~Lee, and J.~B. Ravelomanana.
\newblock The sparse parity matrix.
\newblock {\em Adv. Comb.}, pages Paper No. 5, 68, 2023.

\bibitem{COrank23}
A.~Coja-Oghlan, A.~A. Erg\"ur, P.~Gao, S.~Hetterich, and M.~Rolvien.
\newblock The rank of sparse random matrices.
\newblock {\em Random Structures Algorithms}, 62(1):68--130, 2023.

\bibitem{CDE24}
M.~Collares, J.~Doolittle, and J.~Erde.
\newblock The evolution of the permutahedron.
\newblock {\em arXiv: 2404.17260}, 2024.

\bibitem{DEKK23}
S.~Diskin, J.~Erde, M.~Kang, and M.~Krivelevich.
\newblock Isoperimetric inequalities and supercritical percolation on high-dimensional graphs.
\newblock {\em Combinatorica}, 44(4):741--784, 2024.

\bibitem{DEKK25}
S.~Diskin, J.~Erde, M.~Kang, and M.~Krivelevich.
\newblock Large matchings and nearly spanning, nearly regular subgraphs of random subgraphs.
\newblock {\em arXiv preprint arXiv:2407.16458}, 2024.

\bibitem{DEKK24}
S.~Diskin, J.~Erde, M.~Kang, and M.~Krivelevich.
\newblock Percolation through isoperimetry.
\newblock {\em Annales de l'Institit Henri Poincar\'e B: Probabilit\'es et Statistiques}, to appear.

\bibitem{FK16}
A.~Frieze and M.~Karo\'{n}ski.
\newblock {\em Introduction to random graphs}.
\newblock Cambridge University Press, Cambridge, 2016.

\bibitem{GKSS24}
M.~Glasgow, M.~Kwan, A.~Sah, and M.~Sawhney.
\newblock A central limit theorem for the matching number of a sparse random graph.
\newblock {\em arXiv preprint arXiv:2402.05851}, 2024.

\bibitem{KS81}
R.~M. Karp and M.~Sipser.
\newblock Maximum matching in sparse random graphs.
\newblock In {\em Proceedings of the 22nd Annual Symposium on Foundations of Computer Science}, pages 364--375, 1981.

\bibitem{Kre17}
E.~Kreacic.
\newblock {\em Some problems related to the {K}arp-{S}ipser algorithm on random graphs}.
\newblock PhD thesis, University of Oxford, 2017.

\bibitem{KMMS24}
M.~Krivelevich, T.~M{\'e}sz{\'a}ros, P.~Michaeli, and C.~Shikhelman.
\newblock Greedy maximal independent sets via local limits.
\newblock {\em Random Structures and Algorithms}, 64(4):986--1015, 2024.

\bibitem{L22}
L.~Lichev.
\newblock The giant component after percolation of product graphs.
\newblock {\em Journal of Graph Theory}, 99(4):651--670, 2022.

\bibitem{M98}
C.~McDiarmid.
\newblock Concentration.
\newblock In {\em Probabilistic methods for algorithmic discrete mathematics}, pages 195--248. Berlin: Springer, 1998.

\bibitem{Sal11}
J.~Salez.
\newblock {\em Some implications of local weak convergence for sparse random graphs}.
\newblock PhD thesis, Universit{\'e} Pierre et Marie Curie -- Paris VI; Ecole Normale Sup{\'e}rieure de Paris, 2011.

\bibitem{Sal13}
J.~Salez.
\newblock Weighted enumeration of spanning subgraphs in locally tree-like graphs.
\newblock {\em Random Structures and Algorithms}, 43(3):377--397, 2013.

\bibitem{Hof24}
R.~van~der Hofstad.
\newblock {\em Random {G}raphs and {C}omplex {N}etworks, {V}olume {II}}.
\newblock Cambridge Series in Statistical and Probabilistic Mathematics. Cambridge University Press, 2024.

\bibitem{War16}
L.~Warnke.
\newblock On the method of typical bounded differences.
\newblock {\em Combinatorics, Probability and Computing}, 25(2):269--299, 2016.

\end{thebibliography}

\end{document}